\pdfoutput=1
\documentclass[10pt,a4paper]{article}

\usepackage[utf8]{inputenc}
\usepackage[english]{babel}
\usepackage{amsmath}
\usepackage{amsfonts}
\usepackage{amssymb}
\usepackage{mathrsfs}
\usepackage{amsthm}
\usepackage{latexsym,amsmath,amstext,amssymb,mathtools}
\usepackage{caption}
\usepackage{commath} 
\usepackage{hhline}
\usepackage{listings}
\usepackage{enumerate}
\usepackage{color}
\usepackage{xfrac}
\usepackage{faktor}
\usepackage{colonequals}
\usepackage{tikz-cd}
\usepackage{xr} 
\usepackage[toc,page]{appendix} 
\usepackage{hyphenat}
\usepackage{stmaryrd} 
\usepackage{csquotes}

\usepackage{hyperref}

\usepackage{tikz}
\usepackage{standalone}
\usepackage{forest}

\numberwithin{equation}{section}
\theoremstyle{plain}
\newtheorem{theorem}[equation]{Theorem}
\newtheorem*{theorem*}{Theorem}

\newtheorem{proposition}[equation]{Proposition}
\newtheorem{lemma}[equation]{Lemma}
\newtheorem*{lemma*}{Lemma}
\newtheorem*{proposition*}{Proposition}
\newtheorem*{corollary*}{Corollary}

\newtheorem{corollary}[equation]{Corollary}
\newtheorem{example}[equation]{Example}

\newtheorem{maintheorem}{Theorem}
\newtheorem{introexample}{Example}
\newtheorem{introcorollary}{Corollary}
\newtheorem{introproposition}{Proposition}

\theoremstyle{definition}

\theoremstyle{remark}
\newtheorem*{remark}{Remark}

\newcommand{\dd}{\mathrm{d}}
\newcommand{\BO}{\mathcal{O}}

\newcommand\HH{\mathrm{H}}

\newcommand{\CH}{\mathrm{CH}}
\newcommand{\hh}{\mathrm{h}}
\newcommand{\XX}{\mathcal{X}}
\newcommand{\CC}{\mathbb{C}}
\newcommand{\ZZ}{\mathbb{Z}}
\newcommand{\NN}{\mathbb{N}}
\newcommand{\QQ}{\mathbb{Q}}
\newcommand{\PP}{\mathbb{P}}
\newcommand{\dv}{\partial}

\newcommand\restr[1]{\raisebox{-.5ex}{$|$}_{#1}}
\newcommand{\sing}{\mathrm{sing}}
\newcommand{\Sing}{\mathrm{Sing}}
\newcommand{\Sm}{\mathrm{Sm}}
\newcommand{\codim}{\mathrm{codim}}

\newcommand{\GG}{\mathcal{G}}

\newcommand{\EE}{\mathscr{E}}

\newcommand{\AF}{\mathbb{A}}
\newcommand{\AAA}{\mathbb{A}}

\DeclareMathOperator{\divv}{div}

\newcommand{\sheafhom}{\mathscr{H}\kern -.5pt om}

\DeclareMathOperator{\Sym}{Sym}

\newcommand{\AbVar}{\mathcal{A}}

\newcommand{\Art}{\mathcal{A}^\delta_{t,g-t}}

\newcommand{\spp}{\mathrm{sp}}

\newcommand{\Ker}{\mathrm{Ker}}
\newcommand{\red}{\mathrm{red}}
\newcommand{\mult}{\mathrm{mult}}

\DeclareMathOperator{\Bl}{Bl}

\usepackage{subfiles}

\usepackage{microtype} 
\usepackage[backend=biber, style=alphabetic, doi=false,isbn=false,url=false]{biblatex}
\title{The Gauss Map on Theta Divisors with Transversal $\mathrm{A}_1$ Singularities}
\author{Constantin Podelski}

\begin{document}  
\maketitle 
\begin{abstract}
    We use Lagrangian specialization to compute the degree of the Gauss map on Theta divisors with transversal $\mathrm{A}_1$ singularities. This computes the Gauss degree for a general abelian variety in the loci $\mathcal{A}^\delta_{t,g-t}$ that form some of the irreducible components of the Andreotti-Mayer loci. We also prove that the first coefficient of the Lagrangian specialization is the Samuel multiplicity of the singular locus.
\end{abstract}

\section{Introduction}
The Gauss map relies on the linear nature of abelian varieties by attaching to a smooth point of a divisor, its tangent space translated at the origin. This map was already used by Andreotti \cite{Andreotti1958} in his beautiful proof of the Torelli theorem, and its geometry is intimately connected with the singularities of the theta divisor. Let $\mathcal{A}_g$ be the moduli space of principally polarized abelian varieties of dimension $g$ over the complex numbers. Let $(A,\Theta)\in \mathcal{A}_g$, then the Gauss map
\[ \mathcal{G}_\Theta:\Theta  \dashrightarrow \PP(T_0^\vee A)\simeq \PP^{g-1} \]
is the rational map defined by the complete linear system $|L|$ where $L=\BO_A(\Theta)\restr{\Theta}$ denotes the normal bundle to the hypersurface $\Theta\subset A$. The Gauss map is generically finite if and only if $(A,\Theta)$ is indecomposable as a ppav (see \cite[Sec. 4.4]{Birkenhake2004} for generalities about $\GG_\Theta$). The generic degree of $\GG_\Theta$ is unknown beyond a few cases:
\begin{enumerate}
    \item For smooth Theta divisors, the degree is $[\Theta]^{g}=g!$ (Ex. \ref{Ex: Deg Gauss map smooth theta}).
    \item For non-hyperelliptic (resp. hyperelliptic) Jacobians, the degree is $\binom{2g-2}{g-1}$ (resp. $2^{g-1}$) \cite[247]{arbarello}.
    \item For a general Prym variety the degree is $D(g+1)+2^{g-2}$, where $D(g)$ is the degree of the variety of all quadrics of rank $\leq 3$ in $\PP^{g-1}$ \cite{Verra98}.
\end{enumerate}
Another case where the degree of the Gauss map is straightforward to compute is when $\Theta$ has isolated singularities. In this case we have by \cite[Rem. 2.8]{Gru17}
\[ \deg \GG_\Theta = g! - \sum_{z\in \Sing(\Theta)} \mult_z \Theta \,, \]
where $\mult_z \Theta$ is the Samuel multiplicity as defined for example in \cite[Sec 4.3]{Fulton1998}. For a complex variety $X$ of dimension $n$, we denote by 
\[\chi(X)=\sum_{i=0}^{2n} (-1)^i \dim_\QQ \HH^i(X,\QQ)\]
the topological Euler characteristic. We say that $X$ has \emph{transversal $\mathrm{A}_1$ singularities} if for all points $x\in \Sing(X)$, there is a local analytic isomorphism
\[ (X,x) \simeq (V(x_1^2+\cdots+x_k^2),0)\subset (\CC^{n+1},0) \,, \]
for some $k\geq 2$, where $x_1,\dots,x_{n+1}$ are coordinates on $\CC^{n+1}$. In a sense, this is the simplest kind of singularities to handle after isolated singularities. We have the following:
\begin{maintheorem}[\ref{Theorem: Gauss degree theta divisor with smooth singular locus}]\label{maintheorem: Gauss degree smooth singular locus}
Let $(A,\Theta)\in \AbVar_g$ such that $\Theta$ has transversal $\mathrm{A}_1$ singularities, then
\[ \deg \GG_\Theta = g! - 2 (-1)^{\dim B}\chi(B)-(-1)^{\dim C} \chi(C) \,,\]
where $B=\Sing(\Theta)$ and $C\in| L\restr{B}|$ is a general divisor in the linear system.
\end{maintheorem}

Recently, Codogni, Grushevsky and Sernesi~\cite{Gru17} introduced the stratification of $\mathcal{A}_g$ by the \emph{Gauss loci}
\[ \mathcal{G}^{(g)}_d \coloneqq \{ (A,\Theta)\in \mathcal{A}_g \,|\, \deg \mathcal{G}_\Theta \leq d \} \,. \]
These loci are closed by \cite{KraemerCodogni} and the Jacobian locus $\mathcal{J}_g$ is an irreducible component of $\mathcal{G}^{(g)}_d$ for 
\[ d = \binom{2g-2}{g-1} \,.\]
It is interesting to study how the Gauss loci interact with the stratification introduced by Andreotti and Mayer in \cite{Andreotti1967}, which consists of the loci 
\[ \mathcal{N}^{(g)}_k=\{ (A,\Theta)\in \mathcal{A}_g \,|\, \dim \Sing (\Theta)\geq k \} \,. \]
Andreotti and Mayer prove that the Jacobian locus $\mathcal{J}_g$ is an irreducible component of $\mathcal{N}^{(g)}_{g-4}$. For $g\geq 5$, the known irreducible components of $\mathcal{N}^{(g)}_{g-4}$ away from the locus of decomposable ppav's are by \cite{Donagi88SchottkyProb} and \cite{Debarre1988}:
\begin{itemize}
\item the locus of Jacobians $\mathcal{J}_g$,
\item two loci $\EE_{g,0}$ and $\EE_{g,1}$ arising from Prym varieties of certain étale double covers of bielliptic curves (for a definition see \cite{Debarre1988}),
\item for $2\leq t\leq g/2$, the loci $ \mathcal{A}^{2}_{t,g-t}$ of ppav's containing two complementary abelian varieties of dimension $t$ and $g-t$ respectively, such that the induced polarization is of type $(2)$ (defined by Proposition \ref{PropComplAbVar}).
\end{itemize}
It turns out that by Debarre, a general member of $\mathcal{A}^{2}_{t,g-t}$ for $2\leq t\leq g/2 $ satisfies the conditions of Theorem \ref{maintheorem: Gauss degree smooth singular locus}. As a consequence we have:
\begin{maintheorem}[\ref{Thm: Gauss Degree on A^d_g_1,g_2 }]\label{mainthm: Gauss Degree on A^d_g_1,g_2 }
Let $2\leq t\leq g/2 $, then for a general $(A,\Theta)\in \mathcal{A}^{2}_{t,g-t}$, we have
\begin{align*}
\deg \mathcal{G}_\Theta= t!(g-t)! g \,.
\end{align*}
\end{maintheorem}
In particular, the degree of the Gauss map separates the components $\mathcal{A}^{2}_{t,g-t}$ from $\mathcal{J}_g$. The construction of $\mathcal{A}^{2}_{t,g-t}$ can be generalized for any polarization type $\delta=(a_1,\dots,a_k)$. One has
\[ \mathcal{A}^\delta_{t,g-t} \subset \mathcal{N}^{(g)}_{g-2d} \,, \quad \text{for $\deg \delta \leq t \leq  g/2 $,} \]
where $\deg  \delta \coloneqq a_1\cdots a_k$. Suppose $\delta \in \{(2),(3),(2,2)\}$, let $\deg \delta =d \leq t \leq g/2$, then  $\Art$ is an irreducible component of $\mathcal{N}^{(g)}_{g-2d}$ \cite{Debarre1988}. We compute the degree of the Gauss map for a general member of these loci as well, see Theorem \ref{Thm: Gauss Degree on A^d_g_1,g_2 }. Using different techniques, it is also possible to compute the degree of the Gauss map on the loci $\EE_{g,0}$ and $\EE_{g,1}$, see the forthcoming paper \cite{podelski2023GaussEgt}. 
\par 
The main tool in the proof of Theorem \ref{maintheorem: Gauss degree smooth singular locus} is the notion of Lagrangian specialization, which was already employed by Codogni and Krämer to prove that the Gauss loci are closed \cite{KraemerCodogni}. Let us quickly recall the setup: Let $W$ be a smooth variety. One defines the conormal variety to a closed subvariety $X\subset W$ as the Zariski closure
\[  \Lambda_{X} \coloneqq \overline{\{ (x,\xi)\in T^{\vee}(W) \,|\, x\in \mathrm{Sm}(X)\,, \xi \bot T_x (X) \}} \subset T^\vee W\,. \]
This can be done in a relative setting as well: Let $S$ be a smooth curve, $q:\mathcal{W}\to S$ a smooth morphism and $\mathcal{X}\subset \mathcal{W}$ a subvariety that is flat over $S$. By replacing the tangent spaces in the above definition by the relative tangent spaces over $S$, one obtains the relative conormal variety
\[ \Lambda_{\mathcal{X}/S} \subset T^\vee (\mathcal{W}/S)\,. \]
Let $0\in S$ be a point and $W\coloneqq \mathcal{W}\restr{0}$ and $X\coloneqq \mathcal{X}\restr{0}$ be the fibers above $0$. By \cite{FultonKleimanMacPherson1983}, the specialization of $\Lambda_{\mathcal{X}/S}$ at $0$ is a formal sum of conormal varieties to subvarieties $Z\subset W$, i.e.
\[ \spp_0 (\Lambda_{\mathcal{X}/S})\coloneqq \Lambda_{\mathcal{X}/S}\restr{0}=\sum_{Z\subset W} m_Z \Lambda_Z \,, \]
for some positive integers $m_Z$. We can interpret this in terms of perverse sheaves: Let $\psi_q$ and $\phi_q$ be the nearby and vanishing cycle functors associated to $q$. Suppose that $\XX$ is smooth away from $\XX_0$ and $\dim \XX=n+1$. In that case the Lagrangian specialization computes the characteristic cycle of the nearby cycle functor (see for example \cite[Th. 3.55]{MaximSchurmann2022Survey})
\[ \mathrm{CC}(\psi_q(\underline{\CC}_\XX[n]))=\spp_0(\Lambda_{\mathcal{X}/S}) \,. \] 
Our next result is in a sense the leading term of the Lagrangian specialization in the codimension $1$ case (\ref{Prop: Lagrangian specialization non-reduced special fiber case}):
\begin{introproposition}[Leading term of the Lagrangian specialization]
In the above setting assume that $\mathcal{X}\subset \mathcal{W}$ is of codimension $1$. Then
\[ \spp_0(\Lambda_{\XX/S})=\sum_i (\mult_{X_i} X) \cdot \Lambda_{X_{i,\red}} + \sum_{Z \varsubsetneq X_{i,\red} } m_Z \Lambda_Z  \,, \]
where $X_i$ are the reduced irreducible components of $X$ and $\mult_{X_i} X=\mathrm{len}(\BO_{X,X_i})$ is the geometric multiplicity.
\end{introproposition}
When $X$ is reduced, we can go one step further (\ref{theorem: Second order Lagrangian Specialization}):
\begin{maintheorem}[Second term of the Lagrangian specialization]\label{maintheorem: Second Order Lagrangian Specialization}
In the above setting assume moreover that $X$ is reduced and $\XX\restr{s}$ is smooth for $s\neq 0$. Let $\Sing(X)=\cup_i Z_i$ be the decomposition of the singular locus into its scheme-theoretic irreducible components. Then
\[ \spp_0 \Lambda_{\XX/S}=\Lambda_X+\sum_i (\mult_{Z_i} X ) \cdot \Lambda_{Z_{i,red}} + \sum_{Y\varsubsetneq Z_{i,\red}} m_{Y} \Lambda_{Y} \]
where $\mult_Z X$ is the Samuel multiplicity of $Z$ in $X$ as defined in \cite[Sec. 4.3]{Fulton1998}. 
\end{maintheorem}
Let us consider the setting of plane curve singularities to illustrate the above theorem.
\begin{introexample}
    Let $f:\mathbb{A}^2\to \mathbb{A}$ be a map with an isolated critical point at $0$. Let $C_t\coloneqq \{f=t\}\subset \mathbb{A}^2$. Let $\mu, \kappa, m$ be the Milnor number, Samuel multiplicity (of $\Sing(C_0)$ in $C_0$), and the order of vanishing of $f$ at $0$ respectively. We then have
    \begin{align*}
        \mathrm{CC}(\phi_f(\underline{\QQ}_{\mathbb{A}^2}[1]))&=\mu \Lambda_x\,, \\
        \mathrm{CC}(\underline{\QQ}_{C_0}[1])&=\Lambda_{C_0}+(m-1)\Lambda_x\,.  \\
        \mathrm{CC}(\psi_f(\underline{\QQ}_{\mathbb{A}^2}[1])&=\Lambda_{C_0}+\kappa \Lambda_x \,, 
    \end{align*}
    where the two first equalities can be deduced for example from \cite[Ex. 3.58]{MaximSchurmann2022Survey} and the last equality is a consequence of the above theorem. We thus obtain the well-known relation (see for instance \cite{Nguyen2016InvaPlaneCurves})
    \[ \kappa=\mu+m-1\,. \]
\end{introexample}
As a corollary to Theorem \ref{maintheorem: Second Order Lagrangian Specialization} we obtain an upper bound on the degree of the Gauss map in terms of the degree of the conormal variety to the singular locus (\ref{Cor: bound on degree Gauss map by Lag specialization}):
\begin{introcorollary}\label{Maincor: bound on degree Gauss map by Lag specialization}
    Let $(A,\Theta)\in \AbVar_g$ and let $\Sing(\Theta)=\cup_i Z_i $ be the decomposition of the singular locus of $\Theta$ into its scheme-theoretic irreducible components. We have
    \[ \deg \GG_\Theta \leq g!-\sum_i (\mult_{Z_i} \Theta )\deg (\Lambda_{Z_{i,\red}} )\,,\]
    where
    \[ \deg \Lambda_Z\coloneqq [\Lambda_Z] \cdot [W]\in \HH_0(T^\vee W,\ZZ)\]
    is the degree of the intersection with the zero section $W\hookrightarrow T^\vee W $.
\end{introcorollary}
It is interesting to compare this with the formula obtained by Codogni, Grushevski and Sernesi in \cite[Cor. 2.6]{Gru17}:
\[ \deg \GG \leq g!- \sum_i (\mult_{Z_i} \Theta) \deg (L\restr{Z_{i,\red}})\,,\]
where $\deg(L\restr{Z})=c_1(L)^{\dim Z}\cap [Z]$ is the degree of the polarization $L=\BO_A(\Theta)$ restricted to $Z$. Although the similarity is striking, there is no obvious relation between both bounds: Indeed, let $2\leq t \leq g/2$ and let $(A,\Theta)\in \mathcal{A}^{2}_{t,g-t}$ be general. Then $\Sing(\Theta)$ is smooth and by \ref{Lem: Euler characteristic of B and C} we have
\[ \deg \Lambda_{\Sing(\Theta)}=t!(g-t)! (t-1)(g-t-1) \,. \]
A direct computation using \ref{Thm: Debarre: Cases where Ard verifies Star} shows
\begin{align*} \deg L\restr{\Sing(\Theta)} &=t!(g-t)! \binom{g-4}{t-2} \,,\end{align*}
which is less than $\deg \Lambda_{\Sing(\Theta)}$ for small values of $t$ and greater than $\deg \Lambda_{\Sing(\Theta)}$ for big values of $t$. The proof of Codogni, Grushevski and Sernesi relies on Vogel cycles and it is not clear how both techniques relate. It would also be interesting to know how the next coefficients in the Lagrangian specialization relate to known invariants of the singularity.  \par 

The text is organized as follows: In Section \ref{Section: Lagrangian Specialization}, we recall some well-known facts about the Lagrangian specialization, and prove Theorem \ref{maintheorem: Second Order Lagrangian Specialization}. In Section \ref{Sec: Theta Divisors with a Smooth Singular Locus} we then prove Theorem \ref{maintheorem: Gauss degree smooth singular locus}. Finally, in Section \ref{Sec: Gauss degree on Ard} we prove Theorem \ref{mainthm: Gauss Degree on A^d_g_1,g_2 }, and analyse the result numerically. \par
Throughout this paper we work over the field of complex numbers.

\section{Lagrangian Specialization}\label{Section: Lagrangian Specialization}

\subsection{Generalities on Lagrangian Specialization}
We recall some facts about Lagrangian specialization, see \cite{KraemerCodogni} for an introduction. Let $W$ be a smooth variety (i.e. an integral scheme over $\CC$) of dimension $n$. To a closed subvariety $X\subset W$ we associate its \emph{conormal variety}
\[ \Lambda_{X} \coloneqq \overline{\{ (x,\xi)\in T^{\vee}W \,|\, x\in \mathrm{Sm}(X)\,, \xi \bot T_x (X) \}}\subset T^\vee W \,. \]
The degree of a conormal variety is defined by
\[ \deg(\Lambda_{X}) \coloneqq \deg([\Lambda_{X/S}] \cdot [W ])\,, \]
where $W\hookrightarrow T^\vee (W)$ is embedded as the zero section and the product is in the Chow ring of $T^\vee W$. The same construction can be carried out in families: Suppose $S$ is a smooth (quasi-projective) curve, and
\[ q: \mathcal{W} \to S \]
is a smooth dominant morphism of varieties. Let $\mathcal{X}\subset \mathcal{W}$ be a closed subvariety, flat over $S$. One defines the \emph{relative conormal variety} to $\XX$ as the closure
\[ \Lambda_{\XX/S}\coloneqq  \overline{ \{ (x,\xi)\in T^\vee (\mathcal{W}/S) \, | \, x\in \Sm(\XX/S),\, \xi \bot T_x \XX_{f(x)} \} }\subset T^\vee (\mathcal{W}/S)\,, \]
where
\[ T^\vee(\mathcal{W}/S) \coloneqq T^\vee\mathcal{W} / f^{-1} T^\vee(S) \]
is the relative cotangent bundle. Let
\[ \mathscr{L}(\mathcal{W}/S)= \bigoplus_{\XX \subset \mathcal{W}} \mathbb{Z} \cdot \Lambda_{\XX/S} \]
denote the free abelian group generated by relative conormal varieties to closed subvarieties $\XX\subset \mathcal{W}$ that are flat over $S$. The Lagrangian specialization of $\Lambda_{\XX/\mathcal{W}}\in \mathscr{L}(\mathcal{W}/S)$ is the intersection with the fiber above $s\in S$. This is again a Lagrangian cycle on $\mathcal{W}_s$ \cite{FultonKleimanMacPherson1983} \cite{Trng1988LimitesDT}
\[\spp_s(\Lambda_{\XX/S})\coloneqq  \Lambda_{\XX/S} \cap T^\vee \mathcal{W}_s = m_{\XX_s} \Lambda_{\XX_s}+ \sum_{Z\subset \Sing(\XX_s)} m_Z \Lambda_Z\,, \]
where $m_{\XX_s},m_Z>0$ and the sum runs over finitely many subvarieties $Z\subset \Sing(\XX_s)$. Moreover, for a general $s\in S$,
\[ \spp_s(\Lambda_{\XX/S})=\Lambda_{\XX_s} \,. \]
\begin{remark}\label{Remark: local nature of Lagrangian specialization}
Note that the definition of the conormal variety and of the Lagrangian specialization are local. Thus, we can compute the coefficients $m_Z$ above locally.
\end{remark}
We define the \emph{projectivised} conormal variety $\PP\Lambda_{Z/S}$ by taking the image in the projectivised cotangent space $\PP T^\vee (\mathcal{W}/S)$. From now on we will assume $\XX \subset \mathcal{W}$ to be of codimension $1$. Recall that the \emph{relative singular locus} is the scheme defined locally by 
\[ \Sing(\XX/S)=V(F, \dv_1 F,\dots,\dv_n F)\subset \mathcal{W} \,, \]
where $F$ is a holomorphic function defining $\XX$ and $\dv_i$ generate the relative tangent space $T(\mathcal{W}/S)$ (recall that $\mathcal{W}$ is smooth, thus $\XX$ is a Cartier divisor). We have the following:
\begin{proposition}\label{Prop: Lagrangian is blowup}
Let $S$ be a curve or a point, and suppose $\XX\subset \mathcal{W}$ is of codimension $1$. There is a canonical identification
\[ \PP \Lambda_{\XX/S}=\Bl_{\Sing(\XX/S)} \XX 
\,. \]
\end{proposition}
\begin{proof}
Let $Z=\Sing(\XX/S)$. The Gauss map is the rational map defined on the relative smooth locus by
\begin{align*}
    \mathscr{G}: \XX &\dashrightarrow \PP T^\vee(\mathcal{W}/S) \\
    x &\mapsto (T_x \XX_{f(x)})^\ast  \,.
\end{align*} 
Locally on some open set $U\subset \mathcal{W}$ there are coordinates $x_1,\dots,x_n$ (and $s$ in the relative case), and a function $F$ defining $\XX$. Then $\PP T^\vee (U/S)=U\times \PP^{n-1}$ and
\[ \mathscr{G}(x)=\left(x, \frac{\dv F}{\dv x_1}(x):\dots:\frac{\dv F}{\dv x_n}(x) \right)\,. \]
Thus $Z$ is exactly the base locus of $\GG$. It is well known that after blowing up $Z$ we can extend $\GG$ to an embedding $\tilde{\GG}$ such that the following diagram commutes \cite[Sec 4.4]{Fulton1998}
\begin{center}
    \begin{tikzcd}
    {\Bl_Z \XX} \arrow[d,"p"'] \arrow[dr,hook,"\tilde{\mathscr{G}}"] & \\
    {\XX} \arrow[r,dashrightarrow,"\mathscr{G}"']& {\PP T^\vee (\mathcal{W}/S)}
    \end{tikzcd}\,.
\end{center}
This completes the proof as 
\[\PP \Lambda_{\XX/S}=\overline{\mathscr{G}(\XX)}=\tilde{\mathscr{G}}(\Bl_Z \XX)\simeq \Bl_Z \XX \,. \]
\end{proof}
In the case of ppav's we have the following:
\begin{proposition}\label{Prop: Degree Lagrangian is Gauss degree}
Let $(A,\Theta)$ be a polarized abelian variety, with $\Theta$ reduced and let $\mathcal{G}:\Theta \dashrightarrow \PP T^\vee_0 A$ be the Gauss map. Then
\[ \PP \Lambda_\Theta = \Gamma_\GG \subset A\times \PP T^\vee_0 A \simeq \PP T^\vee A\,,\]
where $\Gamma_\GG$ is the closure of the graph of $\GG$. In particular:
\[ \deg \Lambda_\Theta= \deg \mathcal{G} \,. \]
\end{proposition}
\begin{proof}
The first part of the proposition follows immediately from the proof of Proposition \ref{Prop: Lagrangian is blowup}. Let $v\in T^\vee_0 A$ be general. Then
\begin{align*} 
[\Lambda_\Theta]\cdot[A\times \{0\}]&=[\Lambda_\Theta] \cdot [A\times \{v\}]\\
&=[\PP \Lambda_\Theta]\cdot[A\times \{\overline{v}\}]\\
&=\deg \GG\,.
\end{align*}
\end{proof}
\begin{example}\label{Ex: Deg Gauss map smooth theta}
    Let $(A,\Theta)\in\mathcal{A}_g$ and suppose that $\Theta$ is smooth, then
    \[ \deg \Lambda_\Theta= \deg \GG= g! \,.\]
    Recall that the Gauss map corresponds to the complete linear system $|L\restr{\Theta}|$ with $L=\BO_A(\Theta)$. If $\Theta$ is smooth, the Gauss map is defined everywhere thus
    \[ \deg \GG= [\Theta]^{\cdot g}=g! \]
    by Riemann Roch.
\end{example}
We end the section with the following computation of Lagrangian specialization:
\begin{example}\label{Ex: Lagrangian Specialization, rank $d$ singular quadric case}
Let $n\geq 3$ and consider the following deformation
\[ \begin{tikzcd} 
\XX=\{ x_1^2 + \cdots + x_{n-1}^2 + x_{n} s =0 \}\arrow[r,phantom,"\subset"] \arrow[dr] & \AF^n\times \AF \arrow[d,"q"] \\
& \AF 
\end{tikzcd}\,.
\]
Then
\[ \mathrm{sp}_0 \Lambda_{\XX/S} = \Lambda_{X}+2\Lambda_B+\Lambda_C \,, \]
where 
\begin{align*} 
    B&=\{s=x_1=\cdots=x_{n-1}=0 \} \quad \text{is the singular locus of $X=\XX_0$,} \\ 
  C&=\{s=x_1=\cdots=x_{n}=0 \} \quad \text{is the singular locus of $\XX$.}
\end{align*}
\end{example}
As this example is central in our proof, we will do the computation: By \ref{Prop: Lagrangian is blowup} we have
 \begin{align*}
 \Lambda_\XX &= \mathrm{Blow}_B \XX \\
 &=V(u_1 x_1+\cdots+u_{n} x_n, \, x_i u_j-x_j u_i,\, x_i u_n- s u_i)_{1\leq i,j\leq n-1}\\
 &\subset \XX \times \PP^{n-1} \,,
 \end{align*}
where $u_i$ are homogeneous coordinates on $\PP^{n-1}$. Specializing to $s=0$ and restricting to the open set $U=\{u_1\neq 0\}$ we have (with $a_i=u_i/u_1$ for $2\leq i \leq n$)
\begin{align*}
    \spp_0 \Lambda_\XX \restr{U} &= V\left(x_1^2(1+a_2^2+\cdots+a_{n-1}^2),\right.\\
    & \qquad \qquad \left. x_1(1+a_2^2+\cdots+a_{n-1}^2)+a_nx_n,\,\,
    x_1a_n\right)\\
    & \subset \AAA^2 \times \AAA^{n-1} \,.
\end{align*}
Thus $\spp_0 \Lambda_\XX$ has $3$ irreducible components:
\begin{align*}
    \Lambda_X&=V(1+a_2^2+\cdots + a_{n-1}^2, a_n) \subset \AAA^2\times \AAA^{n-1}  &\text{with multiplicity $1$,}\\
    \Lambda_B&=V(x_1,a_n) & \text{with multiplicity $2$,} \\
    \Lambda_C&=V(x_1,x_n) & \text{with multiplicity $1$.}
\end{align*}

\subsection{First Coefficients in the Lagrangian Specialization}
Let $q:\mathcal{W}\to S$ be a smooth morphism to a quasi-projective curve $S$. Let $\XX\subset \mathcal{W}$ be a variety of codimension $1$, flat over $S$. Let $0\in S$ be a point and $X=\XX_0$, $W=\mathcal{W}_0$ be the special fibers
\[
\XX \subset \mathcal{W} \overset{q}{\longrightarrow} S\,.\]
We have the following:
\begin{proposition}[Zeroth-Order Approximation of the Lagrangian Specialization]\label{Prop: Lagrangian specialization non-reduced special fiber case}
Let $X=\cup_i X_i$ be the scheme-theoretic irreducible components of $X$. We have
\[ \spp_0(\Lambda_{\XX/S})=\sum_i (\mult_{X_i} X) \cdot \Lambda_{X_{i,\red}} + \sum_{Z \varsubsetneq X_{i,\red} } m_Z \Lambda_Z  \,, \]
where $\mult_{X_i} X=\mathrm{len}(\BO_{X,X_i})$ is the geometric multiplicity and the last sum runs over subvarieties $Z\subset \Sing(X_{\red})\cup (\Sing(\XX)\cap X)$.
\end{proposition}
\begin{proof}
By the principle of Lagrangian specialization \cite[Lem 2.3]{KraemerCodogni}, we have
\begin{equation}\label{Equ: Prop non-reduced fiber Lag specialization} \spp_0(\Lambda_{\XX/S})=\sum_i m_{X_i} \Lambda_{X_{i,\red}}+\sum_{Z\subset \Sing(X)} m_Z \Lambda_Z \,,
\end{equation}
for some coefficients $m_{X_i}$, $m_Z$. The definition of the coefficients $m_{X_i}$ is local, thus we can assume that we are working on an affine neighborhood where $X_\red$ is smooth and irreducible. Let $x_1,\dots, x_n,s$ be coordinates on $\mathcal{W}$ such that $q$ is the projection onto $s$. $\XX$ is defined locally by a function $F(x_1,\dots,x_n,s)$. We will show that the ideal of the relative singular locus $\Sing(\XX/S)$
\[ I\coloneqq \left\langle F ,\frac{\dv F}{\dv x_i} \right\rangle_{1\leq i \leq n} \]
is locally principal in the affine coordinate ring of $\XX$ away from a strict subset $Z\subseteq X\cap \Sing(\XX)$. If $X$ is reduced, then it is smooth and $I=\langle 1 \rangle$ so there is nothing to prove. We assume now that $X$ is non-reduced. $X$ is a Cartier divisor in $W$, thus defined by the vanishing of $f^k$ for some $k\geq 2$, where $X_\red$ is defined by the vanishing of $f(x_1,\dots,x_n)$. We have
\[ F=f^k+s^l\cdot g \,, \]
for some function $g$ defined on $\mathcal{W}$ not divisible by $s$, and $l\geq 1$. $g$ does not vanish identically on $X_\red$, else $g$ would be divisible by $f$ and $\XX$ would not be integral. Notice that $V(g\restr{X})\subset \Sing(\XX)\cap X$ thus we can restrict to $\{g\neq 0\}$ and assume $g$ is a unit. We have
\begin{align*}
    I&=\left\langle f^k+s^l\cdot g , \frac{\dv f}{\dv x_i} f^{k-1}+ s^l \frac{\dv g}{\dv x_i} \right\rangle_{1\leq i \leq n}\,.
\end{align*}
As $X_\red$ is smooth, we have $\langle f, \dv_i f \rangle=\langle 1 \rangle$, thus $(f^{k-1}+s^l\cdot h)\in I$ for some function $h$. Thus $s^l(g-fh)\in I$. As $g-fh$ is a unit near $X$ after restricting to a smaller open set we can assume
\[ s^l\in I\,, \quad \text{thus} \quad I=\langle f^{k-1},F\rangle \,. \]
In particular, $\Sing(\XX/S)$ is principal in $\XX$ (defined by $f^{k-1}$), thus by Proposition \ref{Prop: Lagrangian is blowup} we have
\begin{align*}
    \spp_0(\Lambda_{\XX/S})&=(\Bl_{\Sing{\XX/S}} \XX)\restr{0} \\
    &\simeq \XX\restr{0}\\
    &=X\\
    &= \mathrm{len}(\BO_{X,X_\red}) \cdot X_{\red}\\
    &\simeq\mathrm{len}(\BO_{X,X_\red}) \cdot \Lambda_{X_{\red}} \,. 
\end{align*} 
This proves the claim on the coefficients of the $\Lambda_{X_{i,\red}}$. As we only needed to restrict to complements of closed sets in $\Sing(X_\red)\cup \Sing(\XX)$ during the proof, we have that every other cycle $\Lambda_Z$ in the specialization must verify $Z\subset \Sing(X_\red)\cup (\Sing(\XX)\cap X)$.
\end{proof}

When the special fiber is reduced, we can go one step further:
\begin{theorem}[First-Order Approximation of the Lagrangian Specialization]\label{theorem: Second order Lagrangian Specialization}
Assume that $X$ is reduced and $\XX\restr{s}$ is smooth for $s\neq 0$. Let $\Sing(X)=\cup_i Z_i$ be the decomposition of the singular locus into its scheme-theoretic irreducible components. Then
\[ \spp_0 \Lambda_{\XX/S}=\Lambda_X+\sum_i (\mult_{Z_i} X ) \cdot \Lambda_{Z_{i,red}} + \sum_{Y\varsubsetneq Z_{i,\red}} m_{Y} \Lambda_{Y} \]
where $\mult_Z X$ is the Samuel multiplicity of $Z$ in $X$ as defined for example in \cite[Sec. 4.3]{Fulton1998}. 
\end{theorem}
\begin{remark}
    If the singular locus is $0$-dimensional, this computes the full Lagrangian specialization.
\end{remark}
\begin{proof}
By Proposition \ref{Prop: Lagrangian specialization non-reduced special fiber case}, we have
\[ \spp_0 \Lambda_{\XX/S}=\Lambda_X+\sum_i \left( m_{Z_i} \Lambda_{Z_{i,\red}} + \sum_{Y\subset Z_{i,\red}} m_{Y} \Lambda_{Y}\right) \]
for some coefficients $m_{Z_i}$, $m_Y$. The coefficients $m_{Z_i}$ are defined locally, thus after restricting to an open set we can assume that $Z=\Sing(X)$ is irreducible and
\[ \spp_0 \Lambda_{\XX/S}=\Lambda_X+m_{Z}\Lambda_{Z_{\red}}  \,. \]
Let $\mathcal{Z}\coloneqq \Sing(\XX/S)$. Note that by assumption $\mathrm{Supp}(\mathcal{Z})=\mathrm{Supp}(Z)$ and $Z=\mathcal{Z}\cap X$. By Proposition \ref{Prop: Lagrangian is blowup} we have $\PP \Lambda_{\XX/S}=\Bl_{\mathcal{Z}} \XX$. We have the following two fiber squares, 
\begin{center}
    \begin{tikzcd}
     \spp_0(\PP \Lambda_{\XX/S}) \arrow[r,hook] \arrow[d] & \PP \Lambda_{\XX/S} \arrow[d,"f"] \arrow[r,hookleftarrow, "j"] & E  \arrow[d, "g"] \\
     X \arrow[r,hook] & \XX \arrow[r,hookleftarrow, "i"] & \mathcal{Z} 
    \end{tikzcd}
\end{center}
where the right square is the diagram associated to the blowup $f$ and $E$ is the exceptional divisor. By definition, we have
\[ f^\ast [X] = [\PP\Lambda_{\XX/S}\restr{0}] \eqqcolon \spp_0(\PP\Lambda_{\XX/S})=[\PP\Lambda_X]+m_Z[\PP\Lambda_{Z_\red}] \in \CH^1(\PP\Lambda_{\XX/S})\,. \]
Let $\BO(1)=\BO_{\PP\Lambda_{\XX/S}}(-E)\restr{E}$ denote the canonical bundle on $E$ associated to the blowup. By \ref{Prop: Lagrangian is blowup} and \cite[B.6.9]{Fulton1998} there is a canonical embedding 
\[ \PP\Lambda_X = Bl_Z X \hookrightarrow Bl_\mathcal{Z} \XX = \PP \Lambda_{\XX/S} \]
and the restriction of the exceptional divisor $E'\coloneqq E\cap Bl_Z X$ is the exceptional divisor of $Bl_Z X$. Thus
\begin{align*} 
g_\ast \left(j^\ast [\PP\Lambda_X] \cap c_1(\BO(1))^{d-2}\right) &=g_\ast\left( j^\ast [\Bl_{Z'} X] \cap c_1(\BO(1))^{d-2} \right)  \\
&= g_\ast \left( [E'] \cap c_1(\BO(1))^{d-2} \right)  \\
&= \mult_{Z} X \cdot [Z_\red]\in \CH^0(Z)  \,,
\end{align*}
by \cite[Sec. 4.3]{Fulton1998}. Notice that $\PP\Lambda_{Z_\red}\subset E$. We make the abuse of notation to write $[\PP\Lambda_{Z_\red}]$ for the cycle in $\CH^\bullet(\PP\Lambda_{\XX/S})$ as well as $\CH^\bullet(E)$, when it is clear from the context which Chow ring we mean. Recall $\BO_{\PP\Lambda_{\XX/S}}(E)\restr{E}=\BO(-1)$, thus
\begin{align*}
    g_\ast \left( j^\ast [\PP\Lambda_{Z_\red}] \cap c_1(\BO(1))^{d-2} \right)
    &= g_\ast \left( - [\PP\Lambda_{Z_\red}] \cap c_1(\BO(1))^{d-1} \right) \\
    &=-[Z_\red]\in \CH^0(Z)\,,
\end{align*}
as a generic fiber of $\PP\Lambda_{Z_\red}\to Z_\red$ is a $(d-1)$-plane. By definition $[X]=q^\ast[0]\in \CH^1(\XX)$. Since $\mathcal{Z}$ is supported on a fiber of $q$ we have
\[0= i^\ast q^\ast [0]=i^\ast [X]\,.\]
Putting all of this together we have
\begin{align*}
    0 &= g_\ast \left( g^\ast i^\ast [X] \cap c_1(\BO(1))^{d-2} \right) \\
     &= g_\ast \left( j^\ast f^\ast [X] \cap c_1(\BO(1))^{d-2} \right) \\
    &= g_\ast \left( j^\ast(\Lambda_X+m_Z\Lambda_{Z_\red}) \cap c_1(\BO(1))^{d-2} \right) \\
    &=(\mult_{Z} X - m_Z)[Z_\red] \in \CH^0(Z) \,.
\end{align*}
Thus $m_Z=\mult_{Z} X $, as $\CH^0(Z)=[Z_{\red}] \cdot \ZZ$.
\end{proof}

\subsection{Application to Theta Divisors}
We have the following corollary to Theorem \ref{theorem: Second order Lagrangian Specialization}:
\begin{corollary}\label{Cor: bound on degree Gauss map by Lag specialization}
    Let $(A,\Theta)\in \AbVar_g$ and $\cup_i Z_i = \Sing(\Theta)$ the decomposition of the singular locus of Theta into its scheme-theoretic irreducible components. Then
    \[ \deg \GG \leq g!-\sum_i (\mult_{Z_i} \Theta )\deg (\Lambda_{Z_{i,\red}} )\,,\]
    where $\GG:\Theta \dashrightarrow \PP^{g-1}$ is the Gauss map.
\end{corollary}
\begin{proof}
    Let $(A_S,\Theta_S)$ be a $1$-dimensional deformation of $(A,\Theta)$, i.e. an abelian scheme over a smooth curve $S$ with special fiber $(A,\Theta)$, such that $\Theta_s$ is smooth for general $s$. The degree is invariant in flat families \cite[Prop. 2.4]{KraemerCodogni}, thus
    \begin{align*}
        g!&= \deg \Lambda_{\Theta_s} & \text{(Ex. \ref{Ex: Deg Gauss map smooth theta})} \\
        &= \deg(\spp_0 \Lambda_{\Theta_S/S})& \\
        &= \deg( \Lambda_\Theta)+ \sum_i (\mult_{Z_i} \Theta ) \deg(\Lambda_{Z_{i,\red}}) +\sum_{Z\subset Z_i} m_Z \deg( \Lambda_Z ) & \text{(Thm. \ref{theorem: Second order Lagrangian Specialization})}\\
        &\geq \deg( \GG)+ \sum_i (\mult_{Z_i} \Theta ) \deg(\Lambda_{Z_{i,\red}}) \,.&
    \end{align*}
    The last assertion follows from the fact that $\deg \Lambda_Z \geq 0$ for a subvariety $Z$ of an abelian variety \cite[Lem. 5.1]{KraemerCodogni}.
\end{proof}
\begin{remark}
    If the singular locus of $\Theta$ is $0$-dimensional, there are no other terms in the Lagrangian specialization and one recovers a result of \cite[Rem 2.8]{Gru17}
    \[ \deg \GG= g!- \sum_{z\in \Sing(\Theta)} \mult_z \Theta \,. \]
\end{remark}

\section{Theta Divisors with Transversal \texorpdfstring{$\mathrm{A}_1$}{A1} Singularities}\label{Sec: Theta Divisors with a Smooth Singular Locus}
The idea of the proof of Theorem \ref{Theorem: Gauss degree theta divisor with smooth singular locus} is to deform a given ppav to a ppav with a smooth theta divisor. Using the heat equation verified by theta functions, it is then possible to compute the Lagrangian specialization explicitly. Finally, we use the fact that the degree of Lagrangian cycles is invariant in flat families.

\subsection{Deformation of PPAV's}\label{Sec: Deformation of ppav's}
Let $(A,\Theta)\in \AbVar_g$ and denote by $T_A$ the tangent bundle on $A$. It is well-known that there is a canonical identification between $\HH^0(A,\Sym_2(T_A))$ and infinitesimal deformations of $(A,\Theta)$ \cite{Welters1983} and \cite[Sec. 3]{Ciliberto1999}. Specifically, given 
\[\mathscr{D}=\sum_{i,j} \lambda_{ij} \frac{\dv^2}{\dv z_i \dv z_j}\in \HH^0(A,\Sym_2(T_A))\,,\]
there exists a deformation of $(A,\Theta)$, i.e.\ an abelian scheme over a smooth quasi-projective curve $S$
\[ \Theta_S\subset A_S \to S \,,\]
such that the fiber above $0\in S$ is $(A,\Theta)$. Moreover, locally there are coordinates $(z_1,\dots,z_g,s)$ on $A_S$ such that the map to $S$ is given by the projection onto the last coordinate, and if $\theta$ is the theta-function defining $\Theta$ we have
\begin{equation}\label{Equ: heat equation theta} \mathscr{D}(\theta)=\sum_{i,j}\lambda_{ij} \frac{\dv^2 \theta}{\dv z_i \dv z_j} = \frac{\dv \theta}{\dv s} \,. \end{equation}
We call this a deformation in the $\mathscr{D}$ direction.

\subsection{Computation of the Gauss Degree}
Let $(X,0)=V(f)\subset (\CC^n,0)$ be a hypersurface singularity germ. The \emph{scheme theoretic} singular locus of $X$ is defined as
\begin{equation}\label{Equ: scheme theoretic singular locus} \Sing(X)\coloneqq V\left(f, \frac{\dv f}{\dv x_1},\dots,\frac{\dv f}{\dv x_n} \right)\subset (\CC^n,0)\,,
\end{equation}
where $x_1,\dots,x_n$ are some coordinates on $\CC^n$. We have the following:
\begin{proposition}\label{Prop: Appendix: Hessian max rank is equiv to smooth sing locus}
Let $(X,0)=V(f)\subset (\CC^n,0)$ be a hypersurface singularity germ, and $d= \codim_{\CC^n} \Sing(X)$. The Hessian of $f$
\[ H(f) \coloneqq \left( \frac{\dv^2 f}{\dv x_i \dv x_j} \right)_{1\leq i,j\leq n} \]
is of rank at most $d$. The following conditions are equivalent:
\begin{enumerate}[i)]
    \item $\Sing(X)$ is smooth at $0$.
    \item $H(f)$ is of rank $d$ at $0$.
    \item There is a holomorphic change of coordinates $z_1,\dots,z_n$ such that
    \[ f(z)=z_1^2+\cdots+z_d^2 \,.\]
\end{enumerate}
In this case, we say that $X$ has a transversal $\mathrm{A}_1$ singularity at $0$.
\end{proposition}
\begin{proof}
    $(i \iff ii)$ and $(iii\implies i)$ are trivial. Let us show $(i \text{ and } ii) \implies iii$. After a first change of coordinates we can assume $\Sing(X)=V(x_1,\dots,x_d)$. We apply the Morse Lemma with parameters (recalled below) with $x_1,\dots,x_d$ as parameters and the result follows.
\end{proof}
For a proof of the following Lemma, see for example \cite{zoladek2006monodromy}.
\begin{lemma}[Morse Lemma with parameters]\label{Lemma: Morse with parameters}
Let $f(x;s):(\CC^n\times \CC^k,0)\to (\CC,0)$ be a holomorphic function such that the hessian matrix in the first $n$ coordinates
\[ \mathrm{H}(f)_0= \left( \frac{\dv^2 f}{\dv x_i\dv x_j}(0) \right)_{1\leq i,j \leq n} \]
is non-degenerate. Then there is a local holomorphic change of coordinates $h_s:(\CC^n,0) \to (\CC^n,0)$ such that
\[ f(h_s (y) ; s)= f(0;s)+ \sum_{i=1}^n y_i^2 \,. \]
\end{lemma}
We say that a variety $X$ has transversal $\mathrm{A}_1$ singularities if the equivalent conditions of Proposition \ref{Prop: Appendix: Hessian max rank is equiv to smooth sing locus} hold at every singular point of $X$. We now compute the degree of the Gauss map using the results of the previous section:
\begin{theorem}\label{Theorem: Gauss degree theta divisor with smooth singular locus}
    Let $(A,\Theta)\in \AbVar_g$ be a ppav such that $\Theta$ has transversal $\mathrm{A}_1$ singularities. Let 
    \[C\in |\BO_A(\Theta)\restr{B}| \]
    be any smooth member of the linear system. Then the degree of the Gauss map $\GG:\Theta \dashrightarrow \PP^{g-1}$ is
\[ \deg \GG = g! - 2 (-1)^{\dim B}\chi(B)-(-1)^{\dim C} \chi(C) \,,\]
where $\chi$ denotes the usual topological Euler characteristic.
\end{theorem}

\begin{proof}
Let $d=\mathrm{dim}( B)$. Let $\theta\in \HH^0(A,\BO_A(\Theta))$ be a non-zero section. Let $\dv z_1,\dots,\dv z_g$ be a basis of $\HH^0(A,T_A)$. Consider the linear system on $B$
\[ T\coloneqq \left| \frac{\dv^2 \theta}{\dv z_i \dv z_j}\restr{B} \right|_{1\leq i,j\leq g} \subset |\BO_A(\Theta)\restr{B}| \,. \]
Notice that by Proposition \ref{Prop: Appendix: Hessian max rank is equiv to smooth sing locus}, the Hessian of $\theta$ is of rank $d$, in particular $T$ is base-point free and by Bertini's theorem, a general divisor $C\in T$ is smooth. We have
\[ C=\divv \left( \sum_{i,j} \lambda_{ij} \frac{\dv^2 \theta}{\dv z_i \dv z_j}\restr{B}\right) \]
for some $\lambda_{ij}$. By the previous section, there is a deformation $q:(A_S,\Theta_S)\to S$ in the $\mathscr{D}=\sum_{i,j} \lambda_{ij} \dv_i \dv_j \theta$ direction. By \ref{Equ: heat equation theta}, there are locally coordinates $z_1,\dots,z_g,s$ on $A_S$ such that $q$ is the projection onto the last coordinate and
\[ C= V\left(\frac{\dv \theta}{\dv s}\right) \,.\]
Let $\Lambda_{\Theta_S/S}$ be the relative conormal variety. By the Lemma below and Example \ref{Ex: Lagrangian Specialization, rank $d$ singular quadric case} we have
\[ \spp_0 \Lambda_{\Theta_S/S} = \Lambda_{\Theta_0}+2\Lambda_B+\Lambda_C \,. \]
\begin{lemma}\label{Lemma: Normal Form for Morse functions with a 1-dimensional parameter and smooth critical locus}
Let $f(z;s):(\CC^n \times \CC,0) \to (\CC,0)$ be a holomorphic function such that the singular locus and the critical locus of $f\restr{s=0}$ 
\[ B \coloneqq V(f,\dv_1 f,\dots,\dv_n f,s)\,,\quad   C \coloneqq V(\dv_s f\restr{B})\subset B\,, \]
are smooth. Then there is a local holomorphic change of coordinates $z=h_s(\tilde{z})$ such that 
\[ f(h_s(\tilde{z});s)=  \tilde{z}_1^2+\cdots +\tilde{z}_d^2 + \tilde{z}_{d+1} s \,. \]
\end{lemma}
We prove the Lemma below. By generality of $C$ and smoothness of the Theta divisor of a general ppav (it is also apparent from the normal form above), $\Theta_s$ is smooth for $s\neq 0$. Thus
\begin{align*}
    g!&=\deg \Lambda_{\Theta_s}  &\text{(Ex. \ref{Ex: Deg Gauss map smooth theta})} \\
    &= \deg (\spp_s \Lambda_{\Theta_S/S} ) & \\
    &= \deg (\spp_0 \Lambda_{\Theta_S/S})  &\text{(\cite[Prop. 2.4]{KraemerCodogni})}\\
    &= \deg\left( \Lambda_{\Theta_0}+2\Lambda_B+\Lambda_C \right)\\
    &= \deg (\GG) + 2(-1)^{\dim B} \chi(B) + (-1)^{\dim C}\chi(C)  &\text{(Prop. \ref{Prop: Degree Lagrangian is Gauss degree})}\,.
\end{align*} 
\end{proof}

\begin{proof}[Proof of Lemma \ref{Lemma: Normal Form for Morse functions with a 1-dimensional parameter and smooth critical locus}]
After a change of coordinates in $\CC^n$ we can assume $B=\{z_1=\cdots=z_d=0\}$. We have $f\restr{B}=0$ thus
\[ \frac{\dv^2 f}{\dv z_i \dv z_j}\restr{0}=0 \quad \text{for $1\leq i \leq n$ and $d+1\leq j \leq n$.} \]
Thus the Hessian of $F$ in the first $n$ coordinates is
\[ H(F)_0=
\begin{pmatrix}
H(F\restr{\CC^d})_0 & 0 \\
0 & 0
\end{pmatrix}\,.\] 
    By Proposition \ref{Prop: Appendix: Hessian max rank is equiv to smooth sing locus}, $H(F)$ and thus $H(F\restr{\CC^d})$ is of rank $d$ at $0$. By the Morse Lemma with parameters $(z_{d+1},\dots,z_n,s)$, there is a holomorphic change of coordinates $(z_1,\dots,z_d)=h_{(z',s)}(\tilde{z}_1,\dots,\tilde{z}_d)$ where $z'=(z_{d+1},\dots,z_n)$, such that
\[ f(h_{(z',s)}(\tilde{z}),z',s)=\sum_{i=1}^{d} \tilde{z}_i^2 + f(0,z',s) \,. \]
We have
\[ f(0,z',0)=0 \,, \]
and $\frac{\dv f}{\dv s}\restr{B}$ has a simple $0$ in zero, thus after a change of coordinates (in $z'$), we can assume\[ \frac{\dv f}{\dv s}\restr{B}= z_{d+1} \,. \]
Thus 
\[ f(0,z',s)=s(z_{d+1}+s g(z',s) ) \,, \]
for some holomorphic $g$. Making the coordinate change
\[ \tilde{z}_{d+1}=z_{d+1}+s g(z',s) \,, \]
the lemma follows.
\end{proof}

\section{The Family \texorpdfstring{$\Art$}{Agt}}\label{Sec: Gauss degree on Ard}
We apply Theorem \ref{Theorem: Gauss degree theta divisor with smooth singular locus} to the families $\Art$ studied by Debarre in \cite{Debarre1988}. First we recall the definition and known results about $\Art$. Then we compute the Gauss degree for a general member of these families. Finally we analyse the degree numerically and show that it separates the corresponding components of the Andreotti-Mayer locus.
\subsection{Definition of the Family}
Let $A$ be an abelian variety and $L$ an ample line bundle on $A$. Recall that the type $\delta=(a_1,\dots,a_k)$ of $L$ is defined by
\[ \Ker(\Phi_L)\simeq \bigoplus_{i=1}^k (\ZZ/a_i\ZZ)^2\,, \quad \text{and} \quad  a_i|a_{i+1} \quad \text{for $1\leq i <k$}\,,\]
where $\Phi_L: A \to \hat{A}$ is the polarization induced by $L$. We have the following \cite{Debarre1988}, \cite[Th. 5.3.5]{Birkenhake2004}:
\begin{proposition}[Complementary Abelian Varieties]\label{PropComplAbVar}
Let $(A,\Theta)\in \AbVar_{g}$ and $\delta$ be a polarization type. Suppose there is an abelian subvariety $X \subset A$ of dimension $t$ and the induced polarization $L_X=L\restr{X}$ is of type $\delta$. Then there is a unique abelian subvariety $Y \subset A$ (of dimension $g-t$) such that: 
\begin{enumerate}[a)]
\item The morphism $\pi : X\times Y \overset{i_X+i_Y}{\longrightarrow} A $ is an isogeny.
\item We have
\[ \pi^\star L = L_X \boxtimes L_Y \,, \quad \text{where } \quad L_Y=L\restr{Y} \,. \]
\end{enumerate}
Moreover $L_Y$ is also of type $\delta$. We define $\Art\subset \AbVar_g$ to be the set of ppav's verifying the above conditions.
\end{proposition}
Reciprocally, if $(X,L_X)$, $(Y,L_Y)$ are two abelian varieties of the same type $\delta$, of dimension $t$ and $g-t$ respectively, and $\psi:\Ker (\Phi_{L_X}) \to \Ker (\Phi_{L_Y})$ is an antisymplectic isomorphism, then
\[ A \coloneqq X\times Y/{K} \in \Art \quad \text{where} \quad K\coloneqq \{(x,\psi x)\,|\,x\in K(L_X)\}  \,. \]
Thus $\Art$ is irreducible loci of codimension $t(g-t)$ in $\mathcal{A}_g$ \cite[Sec. 9.3]{Debarre1988}. Clearly $\AbVar^\delta_{t,g-t}=\AbVar^\delta_{g-t,t}$, so from now on we will assume $t\leq g/2$. The loci $\Art$ are all distinct. Let $B_X$ (resp. $B_Y$) be the base locus of $L_X$ (resp. $L_Y$). Recall that by the Riemann-Roch theorem,
\[ \hh^0(X,L_X)=\hh^0(Y,L_Y)=\deg \delta \eqqcolon d \,. \]
Thus for $d\leq t$, the base loci $B_X$ and $B_Y$ are non-empty of codimension at most $d$ in $X$ and $Y$ respectively. Let $s^X_1,\dots,s^X_d$ and $s^Y_1,\dots,s^Y_d$ denote a basis of $\HH^0(X,L_X)$ and $\HH^0(Y,L_Y)$ respectively. Let $s$ be a generator of $\HH^0(A,L)$. Then
\[ \pi^\ast s = \sum_{i,j} \lambda_{ij} s^X_i \boxtimes s^Y_j \]
for some $\lambda_{ij}$. Derivating we have
\begin{align*}
    \dd (\pi^\ast s) = \sum_{i,j} \lambda_{ij}( (\dd s^X_i) \boxtimes s^Y_j +  s^X_i \boxtimes (\dd s^Y_j) )\,, 
\end{align*}
which vanishes on $B_X\times B_Y$. Thus
\begin{align}
\pi (B_X\times B_Y) &\subset \Sing(\Theta) \,, \label{Equ: Sing Locus Art and base locus} \\
 \text{and} \quad  \Art &\subset \mathcal{N}^g_{g-2d} \,.
\end{align}
The main result of Debarre concerning the families $\Art$ is the following:
\begin{theorem}[{\cite[Thm. 10.4 and 12.1]{Debarre1988}}]\label{Thm: Debarre: Cases where Ard verifies Star} Let $\delta \in \{(2),(3),(2,2) \}$, and $d=\deg \delta$.
\begin{enumerate}[i)]
    \item If $t\geq d$, then $\mathcal{A}^{\delta}_{t,g-t}$ is an irreducible component of $\mathcal{N}^{(g)}_{g-2d}$. Moreover for a general $(A,\Theta)\in \Art$, there is equality in \ref{Equ: Sing Locus Art and base locus} and $\Theta$ has transversal $\mathrm{A}_1$ singularities.
       \item If $t= \lfloor d/2 \rfloor$, a general $(A,\Theta)\in \Art$ has smooth theta divisor. In this case, $\deg \GG \left(\Art \right)=g!$.
\end{enumerate}
\end{theorem}

We end this section with a result on the dimension of the fibers of the Gauss map. This is a slight slight improvement on \cite[Thm. 1.1]{AuffarthCodogni2019} (the bound on the dimension is stronger):
\begin{proposition}
Let $(A,\Theta)\in \Art$, $d\coloneqq \deg \delta$, and suppose $2\leq d \leq t \leq g/2 $. Suppose there is a divisor $D\in |L_X|$ such that $D$ is smooth at some point $x\in B_X$. Then some fibers of the Gauss map $\GG:\Theta\dashrightarrow \PP^{g-1}$ are of dimension at least $g-t-d+1$.
\end{proposition}
\begin{proof}
Let $\pi:X\times Y\to A$ denote the isogeny of \ref{PropComplAbVar}. Let $\tilde{\Theta}\coloneqq \pi^\ast \Theta \subset X\times Y$. By \cite[Prop. 9.1]{Debarre1988}, there is a basis $s_1^X,\dots,s_d^X$ (resp. $s_1^Y,\dots,s_d^Y$) of $\HH^0(X,L_X)$ (resp. $\HH^0(Y,L_Y)$), such that
\[ \tilde{\Theta} = \divv s \,, \quad \text{where} \quad s=\sum_{i=1}^d s^X_i \otimes s^Y_i \,. \]
We can assume $\divv(s^X_d)$ is smooth for some $x\in B_X$. Let $F=V(s^Y_1,\dots,s^Y_{d-1})\setminus V(s^Y_d)\subset Y$. For $y\in F$, we have
\begin{align*}  \dd_{x,y} s &= \sum_i \dd_x s^X_i \otimes s^Y_i(y)+ s^X_i(x)\otimes  \dd_y s^Y_i = \dd_x s^X_d \otimes s^Y_d(y) \neq 0\,.
\end{align*}
Moreover for $y\in F$ this a constant conormal vector $v\in \PP T_0^\vee (X\times Y)$. Thus the preimage of $v$ by the Gauss map contains $\{x\}\times F$ which is of dimension at least
\[ \dim Y - (d-1) = g-t-d+1 \,. \]
\end{proof}

\subsection{Gauss Degree on \texorpdfstring{$\Art$}{Agt}}\label{Section: Gauss degree on Art}
Knowing \ref{Theorem: Gauss degree theta divisor with smooth singular locus} and \ref{Thm: Debarre: Cases where Ard verifies Star}, the computation of the Gauss Degree on a general $(A,\Theta)\in \Art$ boils down to a relatively simple Euler characteristic computation:
\begin{lemma}\label{Lem: Euler characteristic of B and C}
Let $(A,\Theta)\in \Art$, and assume that $\Theta$ has transversal $\mathrm{A}_1$ singularities and equality holds in \ref{Equ: Sing Locus Art and base locus}, then
\[\chi(\Sing(\Theta))=(-1)^{g-2d} t!(g-t)!\binom{t-1}{d-1}\binom{g-t-1}{d-1} \,.\]
If $C\in | L\restr{\Sing(\Theta)}|$ is smooth, then
\[ \chi(C)=(-1)^{g-2d-1} t!(g-t)!c_{t-d,g-t-d}\,,\]
where $c_{m,n}$ is defined by the generating series
\[\frac{x+y}{(1-x)^d(1-y)^d(1-x-y)}=\sum_{m,n} c_{m,n} x^m y^n \,. \]
\end{lemma}
\begin{proof}
We keep the notation of the previous section. By assumption there is an isogeny of degree $d^2$, $\pi: B_X\times B_Y\to \Sing(\Theta)$, thus
\[  \chi(\Sing(\Theta))=\chi(B_X\times B_Y)/d^2\,.\]
The Euler characteristic of $B_X$ is the degree of the top Chern class $c_{t-d}(T_{B_X})$ of the tangent bundle. $B_X$ is the complete intersection of $d$ divisors in $|L_X|$ thus $N_{B_X/X}=L_X^{\oplus d}\restr{B_X}$ and by \cite[Ex. 3.2.12]{Fulton1998} and Riemann-Roch we have
\begin{align*}
    \deg(c(T_{B_X}))&=\deg \left(c(T_X)\restr{B_X}\cdot (c(L_X\restr{B_X}))^{-d} \right)\\
    &=\deg \left(1\cdot (1+c_1(L_X))^{-d}\cap [B_X]\right) \\
    &=\deg \sum_{k\geq 0} \binom{d+k-1}{d-1}(-1)^k c_1(L_X)^{k+d}\cap [X]  \\
    &= (-1)^{t-d}d\binom{t-1}{d-1} t! \,.
\end{align*}
The same computation applies to $B_Y$, thus
\[ \chi(B_X\times B_Y)=(-1)^{g-2d}d^2 \binom{t-1}{d-1}\binom{g-t-1}{d-1} t!(g-t)! \,. \]
We now compute $\chi(C)$. Let $C'=\pi^\ast C\subset X\times Y$. Let $x =c_1(p_X^\ast L_X)\in \mathrm{CH}^1(X\times Y)$ and $y= c_1(p_Y^\ast L_Y)\in \mathrm{CH}^1(X\times Y)$. We have $C'\in|(L_X\boxtimes L_Y)\restr{B_X\times B_Y}|$,
\[\left[ C' \right]=x^d y^d(x+y) \in \mathrm{CH}^{2d+1}(X\times Y)\,, \]
and $N_{C'/X\times Y}=\left((p_X^\ast L_X)^{\oplus d}\oplus (p_Y^\ast L_Y)^{\oplus d} \oplus (L_X\boxtimes L_Y)\right)\restr{C'}$. By \cite[Ex. 3.2.12]{Fulton1998} we have
\begin{align*}
     c(T_{C'})&= {c(T_{X \times Y}\restr{C'})}\cdot{c(N_{C'/X\times Y})^{-1}}\\
     &=(1+x)^{-d}(1+y)^{-d}(1+x+y)^{-1} \cap [C'] \\
     &=\frac{x^{d}y^d(x+y)}{(1+x)^d(1+y)^d(1+x+y)} \cap [X\times Y] \\
     &=x^dy^d\sum_{m,n}(-1)^{m+n+1}c_{m,n} x^m y^n\,. 
     \end{align*}
The only term of degree $g$ in the above series which does not vanish is $x^t y^{g-t}$ and
\[\deg ( x^{t} y^{g-t})=d^2t!(g-t)! \]
by Riemann-Roch. Thus
\[ \chi(C)=\chi(C')/d^2=(-1)^{g-2d-1} t!(g-t)! c_{t-d,g-t-d}\,.\]
\end{proof}
We have the following:
\begin{theorem}\label{Thm: Gauss Degree on A^d_g_1,g_2 }
Let $\delta\in \{(2),(3),(2,2)\}$, let $t\geq d \coloneqq \deg \delta$, let $(A,\Theta)\in \Art$ be general and $\GG:\Theta \dashrightarrow \PP^{g-1}$ be the Gauss map. Then
\begin{align*}
\deg \mathcal{G}=  g! - t!(g-t)! a_{t-d,g-t-d}    \,. 
\end{align*}
where $a_{m,n}$ is defined by the generating series
\[ \frac{1}{(1-x)^d(1-y)^d}+\frac{1}{(1-x)^d(1-y)^d(1-x-y)}=\sum_{m,n} a_{m,n} x^m y^n \,. \]
More explicitly,
\[ \deg \GG = t!(g-t)! \left(  \binom{t-1}{d-1} \binom{g-t-1}{d-2} + \sum_{k=2}^d \binom{t-k}{d-k}\binom{g-t-1+k}{d-1} \right) \,. \]
\end{theorem}
\begin{remark}
    The theorem above holds more generally when $(A,\Theta)\in 
    \Art$, $\Sing(\Theta)$ is smooth of dimension $g-2d$ and equality holds in \ref{Equ: Sing Locus Art and base locus}, but we do not know for which values of $\delta$, $t$ and $g$ this happens in general.
\end{remark}
\begin{proof}
By Theorem \ref{Thm: Debarre: Cases where Ard verifies Star} and Theorem \ref{Theorem: Gauss degree theta divisor with smooth singular locus}, there is a smooth divisor $C\in |{L_A\restr{\Theta_{\sing}}} |$ such that 
\begin{align*} 
\deg \GG &= g!-2(-1)^{g-2d}\chi(\Sing(\Theta) )-(-1)^{g-2d-1}\chi(C))\,.
\end{align*}
By \ref{Lem: Euler characteristic of B and C} we have
\begin{align*}
(-1)^{g-2d}\chi(\Sing(\Theta))&=t!(g-t)! \binom{t-1}{d-1}\binom{g-t-1}{d-1} \\
&= t!(g-t)! \left\{ \frac{1}{(1-x)^d(1-y)^d} \right\}_{x^{t-d}y^{g-t-d}}\,,
\end{align*}
Thus
\[ 2\chi(\Sing(\Theta))+\chi(C)=t!(g-t)! a_{m,n} \,, \]
where
\begin{align*}
\sum_{m,n\geq 0} a_{m,n}x^m y^n &=
 \frac{2}{(1-x)^d(1-y)^d}+\frac{x+y}{(1-x)^d(1-y)^d(1-x-y)} \\ &=\frac{1}{(1-x)^d(1-y)^d}+\frac{1}{(1-x)^d(1-y)^d(1-x-y)}\,.
\end{align*}
We use the combinatorial Lemma \ref{Lem: generating series Amn} below to conclude
\begin{align*}
\deg \GG &= g!-t!(g-t)! a_{t-d,g-t-d} \\
   &=g!-t!(g-t)! \left( \binom{t-1}{d-1} \binom{g-t-1}{d-1} + \binom{t+g-t}{t} \right.  \\
    &\quad \left.- \sum_{k=1}^d \binom{t-k}{t-d}\binom{g-t-1+k}{d-1} \right) \\
    &=t!(g-t)! \left( \binom{t-1}{d-1} \binom{g-t-1}{d-2} + \sum_{k=2}^d \binom{t-k}{d-k}\binom{g-t-1+k}{d-1} \right)\,.
\end{align*}
\end{proof}


The generating series of the theorem has the following combinatorial interpretation: 
\begin{lemma}\label{Lem: generating series Amn}
Consider the generating series
\[ \frac{1}{(1-x)^{d}(1-y)^{d}(1-x-y)}= \sum_{m,n\geq 0} A_{m,n} x^m y^n \,. \]
then the coefficient $A_{m,n}$ is equal to the number of (weak) $m+d+1$ compositions of $n+d$
\[ 0 \leq a_1 \leq \cdots\leq  a_{m+d} \leq a_{m+d+1}=n+d \,, \]  such that
$a_{m+1} \geq d$. This number is equal to
\[ A_{m,n}= \binom{m+n+2d}{m+d} - \sum_{k=1}^d \binom{m+d-k}{m} \binom{n+d-1+k}{d-1} \,. \]
\end{lemma}
\begin{proof}
Recall that by a (here we mean weak) $m$ composition of $n$ we mean an $m$-tuple $(a_1,\dots , a_m)$ such that
\[ 0 \leq a_1 \leq \cdots \leq a_{m-1} \leq a_{m} = n \,. \]
The number of $m$ compositions of $n$ is equal to 
\[ \binom{n+m-1}{m-1} \,. \]
We know that
\[ \frac{1}{(1-y)^d}= \binom{n+d-1}{d-1} y^n  \]
is the generating series for the $d$-compositions of $n$. Moreover,
\[ \frac{1}{1-x-y}= \sum_{m,n \geq 0} \binom{m+n}{m} x^m y^n \]
is the generating series for the $m+1$-compositions of $n$. Thus
\[ \frac{1}{(1-y)^d(1-x-y)} \]
is the generating series for the $(m+d+1)$-compositions of $n$. Now we can interpret $1/(1-x)^d$ as the generating series of the $m+1$ compositions of $d-1$. Thus the coefficient $A_{m,n}$ is in bijection with the set
\[ \bigsqcup_{k=0}^m \{ k+1 \text{ composition of } d-1 \} \times \{m-k+1+d \text{ composition of } n \} \,. \]
Now to a $k+1$-composition of $d-1$ $(a_1,\dots , a_{k+1})$ and a $m-k+1+d$-composition of $n$ $(b_1,\dots,b_{m-k+1+d})$, we associate a $m+d+1$-composition of $n+d$ in the following way:
\begin{align*}
    \tilde{a}_i & =a_i \quad \text{for} \quad 1\leq i \leq k \\
    \tilde{a}_i & = b_{i-k} + a_{k+1} + 1 \quad \text{for} \quad k+1 \leq i \leq m+d+1 \,.
\end{align*}
Now it is immediate that this gives a bijection to all the $m+d+1$-compositions of $n+d$ such that $a_{m+1}\geq d$. The inverse map is given by choosing $k+1$ to be the first coefficient of the composition above $d$.\\
Thus
\begin{align*} 
A_{m,n}&= \# \{ m+d+1 \text{ compositions of } n+d \} \\
& \quad - \sum_{k=0}^{d-1} \# \{ m+d+1 \text{ compositions of } n+d \text{ such that } a_{m+1}=k \} \\
&= \# \{ m+d+1 \text{ compositions of } n+d \}  \\
& \quad - \sum_{k=0}^{d-1} \# \{ m+1 \text{ compositions of } k \} \times \{ d \text{ compositions of } n+d-k \} \\
&= \binom{m+n+2d}{m+d} - \sum_{k=0}^{d-1} \binom{m+k}{m} \binom{n+2d-1-k}{d-1}
\end{align*}
\end{proof}

\subsection{Numerical Analysis of the Degree}\label{Sec: Numerical Analysis Degree}
For an irreducible locus $Z\subset \mathcal{A}_g$, we denote by $\deg \GG(Z)$ the degree of the Gauss map for a general $(A,\Theta)\in Z$. We close this section with a numerical analysis of the degree $\deg \mathcal{G} \left( \mathcal{A}^\delta_{t,g-t} \right)$ as $t$ varies in $[\![ d, \lfloor g/2 \rfloor ]\!]$. We have the following:
\begin{proposition}\label{Prop: Variation deg G Ard}
For $\delta\in \{(2),(3),(2,2)\}$ and $g\geq 2d\coloneqq 2\deg \delta$, the degree of the Gauss map on the loci $\mathcal{A}^\delta_{t,g-t}$,
\begin{align*}
    [\![ d , \lfloor g/2 \rfloor ]\!] &\to \NN \\
    t &\mapsto \deg \mathcal{G}\left( \mathcal{A}^\delta_{t,g-t} \right)
\end{align*}
is a strictly decreasing function of $t$. In particular, the degree of the Gauss map separates these loci.
\end{proposition}
\begin{remark}
The proposition states that the degree separates the loci $\mathcal{A}^\delta_{t,g-t}$ for fixed $\delta$, fixed $g$, and varying $t$. One could ask if this still hold when $\delta$ varies, i.e. that the degree of the Gauss map is different on all loci $\mathcal{A}^\delta_{t,g-t}$ for a fixed dimension $g$. This holds at least in low dimensions (up to $g=1000$). This is not true anymore when $g$ varies, as one can check that the lowest pair of $g$'s where we have an equality of degrees is $g=28$ and $g=30$, with
\[ \deg \mathcal{G}\left( \mathcal{A}^{3}_{5,28-5}\right)=\deg \mathcal{G}\left( \mathcal{A}^{2}_{7,30-7}\right)=3908824930919408467968000000 \,. \]
\end{remark}
\begin{proof}
We will prove this by looking at the explicit description of the degree. Recall that by \ref{Thm: Gauss Degree on A^d_g_1,g_2 } the degree is given by
\begin{align*} 
F_g(t)&= t!(g-t)!\left(  \binom{t-1}{d-1} \binom{g-t-1}{d-2} + \sum_{k=2}^d \binom{t-k}{t-d}\binom{g-t-1+k}{d-1} \right) \,. 
\end{align*}
We will now prove the proposition by doing each possible value of $\delta$ separately.
\par \emph{Case $\delta=(2)$}. In this case the formula becomes
\[F_g(t)=t!(g-t)!g \,, \]
and this is obviously a decreasing function of $t$ in the range $2\leq t\leq \lfloor g/2 \rfloor$.
\par \emph{Case $\delta=(3)$}. In this case,
\[ F_g(t)=t!(g-t)!(-t^2+gt+3-g) \,. \]
Let $f(x)=-x^2+gx+3-g$. We have
\begin{align*}
    \Delta F_g(t) &\coloneqq  F_g(t+1)-F_g(t)= t!(g-t-1)!(g-2t-1)h_g(t)\,,
\end{align*}
with $h_g(t)=t^2-(g-1)t+g-2 $. Evaluating we have
\begin{align*}
    h_g(3)&=10-2g<0 \quad \text{for }g\geq 6 \\
    h_g\left(\frac{g-1}{2}\right)&=(-g^2+6g-9)/4 <0 \quad \text{for }g\geq 6  \,.
\end{align*}
$h_g$ is convex, thus strictly negative on $[3,(g-1)/2]$, and so $F_g$ is strictly decreasing.
\par \emph{Case $\delta=(2,2)$}.
Now
\[F_g(t)=t!(g-t)!\frac{g}{12}(2t+1-g) h_g(t)\,, \]
where
\[ h_g(x)=x^4+x^3(-2g+2)+x^2(g^2+g-7)+x(-3g^2+11g-8)+2g^2-10g+12 \,. \]
We compute
\[ \frac{\dv h_g}{\dv x}=( 2 x+1-g) ( 2 x^2+2 x(1 - g )+  3 g -8) \,, \]
which is positive for $4\leq x \leq (g-1)/2$. Evaluating at $x=4$ we have
\[ h_g(4)=6(g^2-13g+42) >0 \quad \text{for } g\geq 8 \,. \]
Thus
\[ \Delta F_g(t) <0 \quad \text{for } 4\leq t \leq \lfloor g/2 \rfloor -1 \quad \text{and } g\geq 8 \,,\]
and thus the degree of the Gauss map is strictly decreasing on this range. 
\end{proof}
Finally, we study how this degree compares with the degree of the Gauss map on Jacobians.
\begin{proposition}\label{Prop: comparison Gauss degree Jacobians and Ard loci}
The degree of the Gauss map on Jacobians is always different than on a general member of the loci $\mathcal{A}^\delta_{t,g-t}$. Namely, for $g\geq 7$, $\delta\in \{(2),(3),(2,2)\}$ and $t\geq d$,
\[ \deg \mathcal{G}\left( \mathcal{A}^\delta_{t,g-t} \right) > \det \mathcal{G} \left( \mathcal{J}_g \right) \,. \]
For $g=5$ or $g=6$ the above inequality fails, but the degrees are still different.
\end{proposition}
\begin{proof}
By Proposition \ref{Prop: Variation deg G Ard}, the lowest term on the left hand side of the inequality is achieved when $\delta=(2)$ and $t= \lfloor g/2 \rfloor$. Thus we have to study
\[ \deg \mathcal{G}\left(\mathcal{A}^{2}_{\lfloor g/2 \rfloor,g- \lfloor g/2 \rfloor}\right)-\deg \mathcal{G}(\mathcal{J}_g )\geq g (g/2)!^2-\binom{2g-2}{g-1}\,. \]
Using Stirlings lower bound for the factorial we have, for $g\geq 22 >8e$
\begin{align*}
    g(g/2)!^2 > (g/2)!^2 > \binom{2g-2}{g-1}\,.
\end{align*}
The remaining values can be checked by hand. We get for example for $g=7$
\[ \deg \mathcal{G}\left(\mathcal{A}^{2}_{3,4}\right)=1008>\det \mathcal{G} \left( \mathcal{J}_7 \right)=924 \,. \]
For $g=6$
\[\deg \mathcal{G}\left(\mathcal{A}^{2}_{3,3}\right)=216 \,, \quad \deg \mathcal{G}\left(\mathcal{A}^{2}_{2,4}\right)=288\,, \quad \deg \mathcal{G} \left( \mathcal{J}_6 \right)=252 \,. \]
For $g=5$

\[\deg \mathcal{G}\left(\mathcal{A}^{2}_{2,3}\right)=60 \,, \quad 
\deg \mathcal{G} \left( \mathcal{J}_5 \right)=70 \,. \]

\end{proof}

\printbibliography

\end{document}